\documentclass[12pt]{amsart}

\usepackage{fullpage}
\usepackage{latexsym,amssymb,amsfonts,amsmath,mathrsfs,float}
\usepackage[font=small,labelfont=bf, width=.618\textwidth]{caption} % custom captions
\usepackage{xcolor}
\usepackage{tikz, graphics, enumerate}
\usepackage{hyperref}

  \renewcommand{\Pr}{\mbox{\rm Pr}}	
  
  \newcommand{\E}{\mathbb{E}}

  \newcommand{\R}{\mathbb{R}} % reals
  \newcommand{\C}{\mathbb{C}} % complex numbers
  \newcommand{\Z}{\mathbb{Z}} % integers
  \newcommand{\st}{:\,} % "such that" to define sets

  \newcommand{\eps}{\varepsilon}

  \DeclareMathOperator{\diag}{diag}

  \newcommand{\beq}{\begin{equation}}
  \newcommand{\eeq}{\end{equation}}
  \newcommand{\beqn}{\begin{equation*}}
  \newcommand{\eeqn}{\end{equation*}}
  \newcommand{\beqr}{\begin{eqnarray}}
  \newcommand{\eeqr}{\end{eqnarray}}
  \newcommand{\beqrn}{\begin{eqnarray*}}
  \newcommand{\eeqrn}{\end{eqnarray*}}
  \newcommand{\bmline}{\begin{multline}}
  \newcommand{\emline}{\end{multline}}
  \newcommand{\bmlinen}{\begin{multline*}}
  \newcommand{\emlinen}{\end{multline*}}

  \theoremstyle{plain}
  \newtheorem{theorem}{Theorem}[section]
  \newtheorem{lemma}[theorem]{Lemma}
  \newtheorem{setting}[theorem]{Setting}
  
  \newtheorem{claim}[theorem]{Claim}

  \theoremstyle{definition}

  \theoremstyle{remark}
  \newtheorem{remark}[theorem]{Remark}
  \renewenvironment{proof}[1][]{
    	\begin{trivlist}
     	\item[\hspace{\labelsep}{\em\noindent Proof#1:\/}]}
     	{{\hfill$\Box$}
    	\end{trivlist}
  }
  
  \makeatletter
  \newtheorem*{rep@theorem}{\rep@title}
  \newcommand{\newreptheorem}[2]{%
  \newenvironment{rep#1}[1]{%
  \def\rep@title{#2 \ref{##1}}%
  \begin{rep@theorem}}%
  {\end{rep@theorem}}}
  \makeatother

  \newreptheorem{lemma}{Lemma}
  \newreptheorem{theorem}{Theorem}
  \newreptheorem{corollary}{Corollary}
  \newreptheorem{proposition}{Proposition}
  \newreptheorem{conjecture}{Conjecture}
  \newreptheorem{example}{Example}

\title{The Littlewood-Offord Problem for Markov Chains}
\author{Shravas Rao}
\thanks{
    This material is based upon work supported by the National Science Foundation Graduate Research Fellowship Program under Grant No. DGE-1342536.
}
\date{\today}

\begin{document}

\maketitle

\begin{abstract}
The celebrated Littlewood-Offord problem asks for an upper bound on the probability that the random variable $\eps_1 v_1 + \cdots + \eps_n v_n$ lies in the Euclidean unit ball, where $\eps_1, \ldots, \eps_n \in \{-1, 1\}$ are independent Rademacher random variables and $v_1, \ldots, v_n \in \R^d$ are fixed vectors of at least unit length.
We extend many known results to the case that the $\eps_i$ are obtained from a Markov chain, including the general bounds first shown by Erd\H{o}s in the scalar case and Kleitman in the vector case, and also under the restriction that the $v_i$ are distinct integers due to S\'ark\"ozy and Szemeredi.
In all extensions, the upper bound includes an extra factor depending on the spectral gap.
We also construct a pseudorandom generator for the Littlewood-Offord problem using similar techniques.
\end{abstract}

\section{Introduction}

Let $v_1, \ldots, v_n \in \R^d$ be fixed vectors of Euclidean length at least $1$, and let $\eps_1, \ldots, \eps_n$ be independent Rademacher random variables, so that $\Pr[\eps_i = 1] = \Pr[\eps_i = -1] = 1/2$ for all $i$.
The celebrated Littlewood-Offord problem~\cite{LO43} asks for an upper bound on the probability,
\begin{equation}\label{eq:main}
\Pr[\eps_1 v_1 + \cdots + \eps_n v_n \in B]
\end{equation}
for an open Euclidean ball $B$ with radius $1$.
This question was first investigated by Littlewood and Offord for the case $d=1$ and $d=2$~\cite{LO43}.
A tight bound of $\binom{n}{n/2}/2^n = \Theta(1/\sqrt{n})$ when $n$ is even, with the worst case being when the vectors are equal, was found by Erd\H{o}s for the case $d=1$ using a clever combinatorial argument~\cite{E45}.
Such bounds can be contrasted with concentration inequalities like the Hoeffding inequality in the scalar case and the Khintchine-Kahane inequality in the vector case, both of which give an upper bound on the probability $\Pr[\|\eps_1 v_1 + \cdots + \eps_n v_n\| \geq k\sqrt{n}]$ for positive $k$.
In contrast, an upper bound on Eq.~\eqref{eq:main} can be considered a form of anti-concentration, that is showing that the random sum is unlikely to be in $B$.

In the case that the $v_i$ are $d$-dimensional vectors, a tight bound up to constant factors of $C/\sqrt{n}$ was found by Kleitman~\cite{K70}, and was improved by series of work~\cite{S83, S85, FF88, TV12}.
In the scalar case, under the restriction that $v_1, \ldots, v_n$ are distinct integers, an upper bound of $n^{-3/2}$ was found by S\'ark\"ozy and Szemeredi~\cite{SS65}.

In this work, we investigate the case in which $\eps_1, \ldots, \eps_n$ are not independent, but are obtained from a stationary reversible Markov chain $\{Y_i\}_{i=1}^{\infty}$ with state space $[N]$ and transition matrix $A$, and functions $f_1, \ldots, f_n: [N] \rightarrow \{-1, 1\}$, using $\eps_i = f_i(Y_i)$.

Let $\mu$ be the stationary distribution for the Markov chain, and let $E_{\mu}$ be the associated averaging operator defined by $(E_{\mu})_{ij} = \mu_j$, so that for $v \in \R^N$, $E_{\mu}v = \E_{\mu}[v]\mathbf{1}$ where $\mathbf{1}$ is the vector whose entries are all $1$.
Like many results on Markov chains, our generalizations will be in terms of the quantity
\[
  \lambda = \|A-E_{\mu}\|_{L_2(\mu) \rightarrow L_2(\mu)}.
\] 
If the $Y_i$ are independent, that is $A = E_{\mu}$, it follows that $\lambda = 0$.
Often, if $\lambda$ is small, the corresponding Markov chain behaves almost as if it were independent.
In particular, there exists a Berry-Esseen theorem for Markov chains~\cite{M96} and various concentration inequalities for Markov chain~\cite{G98, L98, LCP04}.
In all of these cases, there is an extra factor in the bounds in terms of $\lambda$ which disappears if $\lambda = 0$.

We show that the Littlewood-Offord problem can also be generalized to Markov chains with an extra dependence on $\lambda$, for all dimensions.
We additionally consider the one-dimensional case when the scalars are distinct integers.
In all cases, the proof is based off a Fourier-analytic argument due to Hal\'asz~\cite{H87}.

The random variables in all cases are defined in the same way, which we state below.
\begin{setting}\label{setting}
Let $\{Y_i\}_{i=1}^{\infty}$ be a stationary reversible Markov chain with state space $[N]$, transition matrix $A$, stationary probability measure $\mu$, and averaging operator $E_{\mu}$ so that $Y_1$ is distributed according to $\mu$.
Let $\lambda = \|A-E_{\mu}\|_{L_2(\mu) \rightarrow L_2(\mu)}$, and let $f_1, \ldots, f_n: [N] \rightarrow \{-1, 1\}$ be such that $\E[f_i(Y_i)] = 0$ for every $i$.
Then consider the random variables $f_1(Y_1), f_2(Y_2), \ldots, f_n(Y_n)$.
\end{setting}

We obtain the following theorem that upper bounds the probability that the random sum is concentrated on any unit ball.
In the case that the $v_i$ are one-dimensional, the bound is tight up to a factor of $\sqrt{(1-\lambda)/(1+\lambda)}$ in $\lambda$.
Note that the bound depends on the dimension, while in the independent case, there is no dependence on the dimension.

\begin{theorem}\label{thm:highdim}

Assume the setting of~\ref{setting}.
Let $x_0 \in \R^d$ and $R \geq \frac{1}{C\sqrt{d}}$ for some universal constant $C'$.
For every set of vectors $v_1, \ldots, v_n \in \R^d$ of Euclidean length at least $1$,
\begin{equation*}
\Pr[\|f_1(Y_1)v_1+f_2(Y_2)v_2+\cdots +f_n(Y_n)v_n-x_0\|_{\ell_2} \leq R] \leq \frac{C\cdot R\sqrt{d}}{(1-\lambda)\sqrt{n}}.
\end{equation*}
for some universal constant $C$.
\end{theorem}

%If $v_1, \ldots, v_n$ are integers, this implies that Eq.~\eqref{eq:main} is bounded above by $C/((1-\lambda)\sqrt{n})$ for some universal constant $C$, as any unit ball contains a constant number of integers.

%\begin{theorem}\label{thm:equal}
%Assume the setting of~\ref{setting}.
%Then for every $v_1, \ldots, v_n \geq 1$ and $x_0 \in \R$,
%\begin{equation*}
%\Pr[f_1(Y_1)v_1+f_2(Y_2)v_2+\cdots +f_n(Y_n)v_n = x_0] \leq \frac{C}{(1-\lambda)\sqrt{n}}.
%\end{equation*}
%for some universal constant $C$.
%\end{theorem}

In the one-dimensional case, we also consider the restriction that $v_1, \ldots, v_n$ are \emph{distinct} integers.

\begin{theorem}\label{thm:diff}
Assume the setting of~\ref{setting}.
Then for every set of distinct integers $v_1, \ldots, v_n \geq 1$ and $x_0 \in \Z$,
\[
\Pr[f_1(Y_1)v_1+f_2(Y_2)v_2+\cdots +f_n(Y_n)v_n = x_0] \leq \frac{C}{(1-\lambda)^3n^{3/2}}
\]
for some universal constant $C$.
\end{theorem}

%Finally, for general $d$-dimensional vectors we obtain the following theorem.

Finally, we consider a different setting, where rather than choosing $\eps_1, \ldots, \eps_n$ independently, we choose these uniformly at random from a subset $D$ of $\{-1, 1\}^n$ that we can construct explicitly.

\begin{theorem}\label{thm:prglo}
For every $n$, there exists an explicit set $D \subseteq \{-1, 1\}^n$ of cardinality at most $2^{C_1\sqrt{n}}$ for some universal constant $C_1$ such that the following holds.
For every $v_1, \ldots, v_n \geq 1$ and $x_0 \in \R$ and $\eps$ chosen uniformly at random from $D$
\begin{equation*}
\Pr[|\eps_1v_1+\eps_2 v_2+\cdots +\eps_n v_n-x_0| \leq 1] \leq \frac{C}{\sqrt{n}}.
\end{equation*}
for some universal constant $C$ independent of $n$.
\end{theorem}

One interpretation of Theorem~\ref{thm:prglo} is that one can obtain similar results as in the Littlewood-Offord problem in one dimension using much less randomness, and in particular, using $C_1\sqrt{n}$ bits of randomness rather than $n$.

This setting was also considered in~\cite{KKL17}, in which the authors were able to construct an explicit set of cardinality $n 2^{n^{c}}$, from which a random sample satisfies
\[
\Pr[f_1(Y_1)v_1+f_2(Y_2)v_2+\cdots +f_n(Y_n)v_n = x_0] \leq \frac{\log(n)^{C_1/c}}{\sqrt{n}}.
\]
for any constant $c$ bounded above by $1$.
Sampling from the set in Theorem~\ref{thm:prglo} guarantees a stronger bound on the probability that the sum lands in any interval, while requiring more randomness when $c < 1/2$.

\subsection{Future Work}

It would be interesting to remove the dependence on the dimension in Theorem~\ref{thm:highdim}, which does not appear in the tightest bounds for independent random variables.

The setting studied by S\'ark\"ozy and Szemeredi, in which the the $v_i$ are distinct positive integers and the random variables are independent, was the first in a series of work investigating under what conditions Eq.~\eqref{eq:main} can be bounded more strongly.
We call a set $Q \subseteq \R^d$ a generalized arithmetic progression (GAP) of rank $r$ if it can be expressed as
\[
Q = \{v_0+x_1v_1+x_2v_2+\cdots+x_rv_r \st x_i \in \Z, M_i \leq x_i \leq M_i'\}
\]
for some $v_0, \ldots, v_r \in \R^d$, $M_1, \ldots, M_r \in \Z$ and $M_1', \ldots, M_r' \in \Z$.
In a series of works starting with~\cite{TV09} and improved by~\cite{TV10, NV11}, it was shown that when $d=1$, if Eq.~\eqref{eq:main} is bounded above by $n^{-C}$ for all unit-balls $B$, then the set $\{v_1, \ldots, v_n\}$ must be mostly contained in some GAP of rank-$r$, where $r$ depends on $C$.
It would be interesting to see if such an analogue holds when the random variables are chosen from a Markov chain.

It would also be interesting to improve Theorem~\ref{thm:prglo} by constructing explicit sets of cardinality smaller than $2^{C_1\sqrt{n}}$ that achieve similar properties.

%A series of work has also answered the question for general $d$-dimensional vectors, both using combinatorial arguments like that of Erdos~\cite{S83, S85, FF88} and also using techniques from Fourier analysis~\cite{TV12}.
%Unfortunately the latter does not seem to give a straightforward generalization to the Markov chain case, as it (what to say here).  It would be interesting to try to solve the vector-valued version for Markov chains.

\section{Preliminaries}

Given vectors $v, \mu \in \R^{N}$ (typically $\mu$ will be a distribution over $[N]$), we define the $L_{p}(\mu)$-norm by
\[
\|v\|_{L_{p}(\mu)}^p = \sum_{i=1}^N |v_i|^p \mu_i.
\]
Additionally, we let the ${L_p(\mu)} \rightarrow {L_q(\mu)}$-operator norm of a matrix $A \in \R^{N \times N}$ be defined as 
\[\|A\|_{{L_p(\mu)} \rightarrow {L_q(\mu)}} = \max_{v: \|v\|_{L_p(\mu)} = 1} \|Av \|_{L_q(\mu)}.\]
%For simplicity, we will use $\|\cdot\|_p$ to denote $\|\cdot\|_{L_p(\mu)}$ or $\|\cdot\|_{{L_p(\mu)} \rightarrow {L_p(\mu)}}$ when $\mu$ is implied.
Finally, we will use $\ell_p$ in place of $L_p(\mu)$ when $\mu$ is the vector whose entries are all $1$.
Note that in this case, $\mu$ is not a distribution.

For a vector $v$, we let $\diag(v)$ be the diagonal matrix where $\diag(v)_{i, i} = v_i$.

Let $A$ be a stochastic matrix, and let $\mu$ be a distribution for which $A$ is reversible, that is, $\mu_i A_{ij} = \mu_j A_{ji}$.
We let $(E_{\mu})_{ij} = \mu_j$ be the averaging operator on $L_{\infty}(\mu) \rightarrow L_{\infty}(\mu)$.
Note that $E_{\mu}$ is also stochastic and reversible on $\mu$.

\section{The Littlewood-Offord problem for independent random variables}

As warm up, we present the bound in the independent case for $1$-dimensional vectors, or scalars.
These calculations will be used later in the proofs of Theorems~\ref{thm:highdim},~\ref{thm:diff}, and~\ref{thm:prglo},.
This bound was first proved by Erd\H{o}s~\cite{E45} who used a clever combinatorial argument that applies Sperner's theorem.
The proof we present is in spirit, due to Hal\'asz~\cite{H87} and is based on techniques from Fourier analysis.

We start by presenting the following concentration inequality due to Ess\'een~\cite{E66}, which will allow us to upper-bound probabilities.
This inequality is in the spirit of Fourier inversion, but written in a way that can be more readily applied for our purposes.
\begin{theorem}[Ess\'een concentration inequality]\label{thm:esseen}
Let $X \in \R^d$ be a random variable taking a finite number of values.
For $R, \eps > 0$,
\[
\sup_{x_0 \in \R^d} \Pr\left[\|X-x_0\|_{\ell_2} \leq R\right] = O\left(\frac{R}{\sqrt{d}}+\frac{\sqrt{d}}{\eps}\right)^d\int_{\xi \in \R^d\st \|\xi\|_{\ell_2} \leq \eps} \left|\E[\exp(2\pi i\langle \xi, X\rangle)]\right|d\xi.
\]
\end{theorem}

The following bound is implicit in the proof of Proposition 7.18 in~\cite{TV06} and will be used to further bound the quantities obtained from Theorem~\ref{thm:esseen}
\begin{claim}\label{eq:imptv}
Let $v_1, \ldots, v_k \in \R$ be such that $|v_j| \geq 1$ for all $j$.
Then
\[
\int_{-1}^{1} \left(\prod_{j \in k} \left|\cos(2\pi \xi v_j)\right|\right) d\xi
\leq
\frac{C}{\sqrt{|k|}},
\]
for some constant $C$.
\end{claim}

We now prove the bound in the independent case.

\begin{theorem}\label{thm:indrad}
Let $v_1, \ldots, v_n \in \R$ be non-zero, and let $\eps_1, \ldots, \eps_n$ be independent random variables uniform over the set $\{-1, 1\}$.
Then for all $x_0 \in \R$,
\[
\Pr[|\eps_1 v_1 +\cdots+\eps_n v_n-x_0| \leq 1] \leq \frac{C}{\sqrt{n}}.
\]
for some constant $C$ independent of $n$.
\end{theorem}

\begin{proof}
By Theorem~\ref{thm:esseen}, the left-hand side can be bounded above by
\begin{align}
C_1 \int_{-1}^{1} \left|\E[\exp(2\pi i \xi (\eps_1 v_1 +\cdots+\eps_n v_n))]\right|d\xi
&= C_1 \int_{-1}^{1} \prod_{j=1}^n \left|\E[\exp(2\pi i  \xi \eps_j v_j )]\right|d\xi \nonumber \\
&= C_1 \int_{-1}^{1} \prod_{j=1}^n \left|\cos(2\pi  \xi v_j )\right|d\xi \label{eq:fincos}\\
&\leq \frac{C_2}{\sqrt{n}} \nonumber
\end{align}
for some constants $C_1$ and $C_2$.
The first equality follows from the independence of the $\eps_j$, the next equality follows from the fact that $\eps_j$ is uniform over $\{-1, 1\}$ for all $j$, and the subsequent inequality follows from Claim~\ref{eq:imptv}.
\end{proof}

\section{The Littlewood-Offord Problem for Random Variables from a Markov chain}\label{sec:expbasic}

Now we consider the case that $\eps_1, \ldots, \eps_n$ are obtained from a Markov chain.
The proof follows very closely the proof for independent random variables in Proposition 7.18 in~\cite{TV06} which itself is due to Hal\'asz~\cite{H87}.

In order to handle the extra dependencies from the Markov chain, we will use the following technical lemma, which is a straightforward adaptation of a Lemma from~\cite{NaorRR17}.
We include a proof in Appendix~\ref{app:proof}.

\begin{lemma}\label{lem:holderapplicationrefine}
Let $k \ge 1$ be an integer, $u_1, \ldots, u_{k+1} \in \C^{N}$ be $N$-dimensional vectors such that $\|u_i\|_{L_{\infty}(\mu)} \leq 1$, $U_i = \diag(u_i)$,
 and $T_1, \ldots, T_k \in \R^{N \times N}$.
For $s \in \{0, 1\}^k$, let $\overline{s} := (0,s,0) \in \{0, 1\}^{k+2}$ and define $t(s) \subseteq [n]$ to be $t(s) := \{ i  \ : \ \overline{s}_i = \overline{s}_{i-1} = 0\}$.
Then,
\begin{align}\label{eq:splittingj}
\left\| U_1 (T_1+(1-\lambda) E_{\mu}) U_2 (T_2+(1-\lambda) E_{\mu}) U_3\cdots U_k (T_k+(1-\lambda) E_{\mu}) U_{k+1} \mathbf{1} \right\|_{L_1(\mu)} \nonumber
\leq \\
 \sum_{s \in \{0, 1\}^k}\left(\prod_{j: s_j = 1}\|T_j\|_{L_2(\mu) \rightarrow L_2(\mu)}\right)\left( \prod_{j: s_j = 0} (1-\lambda) \right)\left(\prod_{j \in t(s)} \left|\langle u_j, \mu \rangle\right|\right)\; .
\end{align}
\end{lemma}

Before proving Theorem~\ref{thm:highdim}, we first prove the following that will allow us to upper-bound negative moments of binomial random variables.

\begin{claim}\label{claim:combinarg}
Let $x = B(n, p)$ be a binomial random variable with $n$ trials, each with success probability $p > 0$.
Then for all positive integers $d$,
\[
\E\left[\frac{1}{(x+1)^d}\right] \leq \frac{d^{d}}{n^dp^d}.
\]
\end{claim}
\begin{proof}
Note that because $d(i+1) \geq i+d$ for all non-negative $i$, the right-hand side is bounded above by $d^d\E\left[\frac{x!}{(x+d)!}\right]$, where the term inside the expected value can be written as
\begin{align*}
\sum_{i=0}^{n} \binom{n}{i}p^{i}(1-p)^{n-i}\frac{i!}{(i+d)!} &=
\sum_{i=0}^{n} \frac{n!}{(n-i)!(i+d)!}p^{i}(1-p)^{n-i} \\
&=
\sum_{i=0}^{n} \binom{n+d}{i+d}p^{i+d}(1-p)^{n-i}\frac{n!}{(n+d)!p^d} \\
&\leq
\frac{n!}{(n+d)!p^d}.
\end{align*}
The claim follows by noting that $n \leq n+i$ for $1 \leq i \leq d$.
\end{proof}

We start by considering the case of $1$-dimensional vectors, or scalars.
We also consider the case in which at most one-half of the $v_i$ have length less than $1$.
This will allow us to generalize to higher dimensions.
We note that in the case of independent random variables the corresponding statement follows from the usual Littlewood-Offord problem, by conditioning on the $\eps_i$ such that $|v_i| < 1$, for just an increase in the constant factor in the bound.
\begin{lemma}\label{thm:equal}
Assume the setting of~\ref{setting}.
Then for every $v_1, \ldots, v_n \in \R$ such that $|\{i: |v_i| \geq 1\}| \geq n/2$ and $x_0 \in \R$,
\begin{equation*}
\Pr[|f_1(Y_1)v_1+f_2(Y_2)v_2+\cdots +f_n(Y_n)v_n-x_0| \leq 1] \leq \frac{C}{(1-\lambda)\sqrt{n}}.
\end{equation*}
for some universal constant $C$.
\end{lemma}
\begin{proof}
By Theorem~\ref{thm:esseen},
\begin{multline}\label{eq:mainstart}
\Pr[|f_1(Y_1)v_1+\cdots +f_n(Y_n)v_n-x_0| \leq 1] \leq 
\\
C_1\int_{-1}^{1} \left|\E[\exp(2\pi i\xi(f_1(Y_1)v_1+\cdots +f_n(Y_n)v_n))]\right|d\xi
\end{multline}
for some constant $C_1$.
Note that
\begin{equation}\label{eq:esseenapp}
\E[\exp(2\pi i\xi(f_1(Y_1)v_1+\cdots +f_n(Y_n)v_n))] 
= 
\E\left[\prod_{j=1}^n \exp(2\pi i \xi f_j(Y_j)v_i)\right].
\end{equation}
Let $T_j = A-(1-\lambda)E_{\mu}$, let $u_j$ be the vector defined by $u_j(y) = \exp(2\pi i\xi f_j(y)v_j)$ for $y \in [N]$, and let $U_j = \diag(u_j)$.
For $s \in \{0, 1\}^{n-1}$, let $t(s)$ be the set of indices $j$ such that $s_{j-1} = s_{j} = 0$, and also includes $1$ if $s_1 = 0$ and includes $n$ if $s_{n-1} = 0$.
Then the right-hand side of Eq.~\eqref{eq:esseenapp} is bounded above by
\begin{multline*}\label{eq:finb}
\left\| U_1 (T_1+(1-\lambda) E_{\mu}) U_2 (T_2+(1-\lambda) E_{\mu}) U_3\cdots U_{n-1} (T_{n-1}+(1-\lambda) E_{\mu}) U_{n} \mathbf{1} \right\|_{L_1(\mu)} \leq
\\
\sum_{s \in \{0, 1\}^{n-1}}\left(\prod_{j: s_j = 1}\lambda \right)\left( \prod_{j: s_j = 0} (1-\lambda) \right)\left(\prod_{j \in t(s)} \left|\cos(2\pi \xi v_j)\right|\right),
\end{multline*}
where the inequality follows  by Lemma~\ref{lem:holderapplicationrefine} and evaluating $|\langle \mu, u \rangle|$.

Let $t'(s)$ be the set of indices $j \in t(s)$ such that $|v_j|$ is greater than $1$.
When $|t'(s)| = 0$, the corresponding product disappears.
When $|t'(s)| > 0$, we can apply Claim~\ref{eq:imptv}.
Thus, the right-hand side of Eq.~\eqref{eq:mainstart} can be bounded above by 
\begin{equation}\label{eq:intme}
C_1\sum_{s \in \{0, 1\}^{n-1}}\left(\prod_{j: s_j = 1}\lambda \right)\left( \prod_{j: s_j = 0} (1-\lambda) \right) \frac{C_2}{\sqrt{|t'(s)|+1}}.
\end{equation}

Let $r: \{0, 1\}^{n-1} \rightarrow [n-1]$ be defined as
\[
r = \left|\{j: s_j = s_{j+1} = 0 \text{ and } |v_j| \geq 1\}\right|,
\]
so that $r(s) \leq |t'(s)|$ for all $s \in \{0, 1\}^{n-1}$.
Let $\mathbf{s}$ be a random vector from $\{0, 1\}^{n-1}$ so that for each $s \in \{0, 1\}^{n-1}$
\[
\Pr[\mathbf{s} = s] = \left(\prod_{j: s_j = 1}\lambda\right)\left( \prod_{j: s_j = 0} (1-\lambda) \right).
\] 
By the definition of $r$ and $\mathbf{s}$, the right-hand side of Eq.~\eqref{eq:intme} is bounded above by,
\begin{equation*}\label{eq:startofcomb}
%\Pr[f(Y_1)+\cdots +f(Y_n) = x_0] 
%&\leq 
%\frac{1}{2\pi} \sum_{s \in \{0, 1\}^{n-1}} \int_{-\pi}^{\pi} \left(\prod_{j: s_j = 1}\lambda\right)\left( \prod_{j: s_j = 0} (1-\lambda) \right)\left(\prod_{j \in t(s)} |\cos(v_j \omega)|\right) d\omega  \label{eq:maincos}\\
%&\leq
C_1 \E\left[\frac{C_2}{\sqrt{r(\mathbf{s})+1}}\right].
\end{equation*}

We conclude with the following argument.
Let $r' = B(\lfloor n/4\rfloor-1, (1-\lambda)^2)+1$ where $B(n, p)$ denotes a binomial random variable with $n$ trials, each with success probability $p$.
It follows that $r'$ is dominated by $r(\mathbf{s})+1$, and thus
\begin{equation}\label{eq:finalcomp}
\E\left[\frac{C}{\sqrt{r(\mathbf{s})+1}}\right] \leq  \E\left[\frac{C}{\sqrt{r'}}\right] \leq \left(\E\left[\frac{C^2}{r'}\right]\right)^{1/2},
\end{equation}
where the second inequality follows by Jensen's inequality.
Finally, by Claim~\ref{claim:combinarg}, the right-hand side of Eq.~\eqref{eq:finalcomp} is bounded above by $C\left((1-\lambda)\sqrt{\lfloor n/4\rfloor}\right)^{-1}$ as desired.
\end{proof}

Before proving Theorem~\ref{thm:highdim}, we prove the following bound on random unit vectors.

\begin{claim}\label{claim:forrandomrot}
Let $v \in \R^d$ be a random unit vector uniform over the $d-1$-dimensional sphere.
Then there exists a constant $C$ such that
\[
\Pr\left[|v_1| \geq \frac{1}{C\sqrt{d}}\right] \geq \frac{1}{2}
\]
\end{claim}
\begin{proof}
We start by noting that the probability density function of $v_1$ at $t$ is proportional to $(1-t^2)^{(d-3)/2}$, which is also the probability density of the beta distribution, shifted so that the domain is $[-1, 1]$.
The probability density function at all points is bounded above by
\[
\frac{1}{2^{d-3}}\cdot \frac{\Gamma(d-1)}{\Gamma((d-1)/2)^2} \leq \frac{1}{2^{d-3}} \cdot  \frac{C_1(d-1)^{d-3/2}e^{-d+2}}{C_1^2((d-1)/2)^{d-2}e^{-d+1}} \leq C_2 \sqrt{d-1}
\]
for some constants $C_1$ and $C_2$, where the inequality follows from Stirling's approximation (see~\cite{J15}).
The claim follows by letting $C = C_2/4$.
\end{proof}
We now use Lemma~\ref{thm:equal} to prove Theorem~\ref{thm:highdim} as follows.

\begin{proof}[ of Theorem~\ref{thm:highdim}]
Let $A \in \mathcal{SO}(d)$ be a random rotation uniform over the Haar measure of the special orthogonal group.
Then it is enough to consider the random variable $\|Af_1(Y_1)v_1+\cdots+Af_n(Y_n)v_n-Ax_0\|_{\ell_2}$.
Additionally, the left-hand side in the statement of the theorem is bounded above by
\begin{equation}\label{eq:firstcoord}
\Pr[\left|(Af_1(Y_1)v_1+\cdots+Af_n(Y_n)v_n-Ax_0)_1\right| \leq R].
\end{equation}
This is because if the absolute value of the first coordinate of the random vector is greater than $R$, so is the Euclidean norm.

By Claim~\ref{claim:forrandomrot}, for any fixed $d$, it holds that $|f_i(Y_i)v_i| \geq 1/(C'\sqrt{d})$ for at least half of the $i$ for some constant $C'$.
By Lemma~\ref{thm:equal}, we have that Eq.~\eqref{eq:firstcoord} is bounded above by
\[
C'\cdot R\sqrt{d}\sup_{x_0 \in \R}\Pr\left[\left|(Af_1(Y_1)v_1+\cdots+Af_n(Y_n)v_n-x_0)_1 \right| \leq \frac{1}{C'\sqrt{d}}\right] \leq \frac{C\cdot R\sqrt{d}}{(1-\lambda)\sqrt{n}}
\]
as desired.
\end{proof}

\begin{remark}
In the case of one dimension, Theorem~\ref{thm:highdim} is tight up to a factor of $\sqrt{(1-\lambda)/(1+\lambda)}$.
To see this, consider the transition matrix on two states defined by 
\[
A =     \begin{pmatrix}
    \frac{1-\lambda}{2} & \frac{1+\lambda}{2} \\
    \frac{1+\lambda}{2} & \frac{1-\lambda}{2} 
  \end{pmatrix}
\]
with $f(1) = 1$ and $f(2) = -1$, and stationary distribution uniform over both states.
Such a Markov chain can be interpreted as first choosing a state at random, and then at each subsequent step choosing a new state uniformly at random with probability $1-\lambda$, or switching states with probability $\lambda$.
We can associate with this walk a sequence of numbers, $(X_1, X_2, \ldots)$ obtained as follows.
Whenever a state is chosen at random, we add a new entry in the sequence starting at $1$, and increase this entry every time the state is switched.
Then conditioned on this sequence, $f(Y_1)+f(Y_2)+\cdots+f(Y_n)$ is distributed as $\eps_1+\eps_2+\cdots+\eps_{\mathbf{n}}$ where $\mathbf{n}$ is the number of entries in the sequence that are odd.
Thus, if $\mathbf{n}$ is considered as a random variable,
\[
\Pr[f(Y_1)+f(Y_2)+\cdots +f(Y_n) = 0] \leq \mathbb{E}\left[\frac{C}{\sqrt{\mathbf{n}}}\right]
\]
If we assume that $n$ is large, then the probability that any given step in the walk is the start of a entry that will eventually be of odd length is approximately $1/(1+\lambda)$, and thus, $\mathbf{n}$ is approximately distributed like $B(n, (1-\lambda)/(1+\lambda))$, and thus
\[
\mathbb{E}\left[\frac{C}{\sqrt{\mathbf{n}}}\right] \leq \frac{C}{\sqrt{(1-\lambda)n/(1+\lambda)}}
\]
\end{remark}

\section{Extension to distinct $v_i$'s}

Theorem~\ref{thm:indrad}, the bound obtained in the independent case, is tight when $v_1 = \cdots = v_n = 1$.
It is reasonable to ask if one can obtain better bounds on the probability $\Pr[\eps_1 v_1 + \cdots + \eps_n v_n \in B]$ under certain restrictions of $v_1, \ldots, v_n$.
In particular, when the $v_i$ are distinct integers, S\'ark\"ozy and Szemeredi~\cite{SS65} showed that for all $x_0$ and for some constant $C$
\begin{equation}
\Pr[\eps_1 v_1 + \cdots + \eps_n v_n = x_0] \leq \frac{C}{n^{3/2}}, \label{eq:eqdiff}
\end{equation}
which is a factor $n$ smaller than Theorem~\ref{thm:indrad}.

Like Erd\H{o}s's proof of Theorem~\ref{thm:indrad}, the proof of the above by S\'ark\"ozy and Szemeredi uses a clever combinatorial argument.
However, Hal\'asz's Fourier-analytic argument can also be used to prove the above.
We prove a similar bound in the case of Markov chains.

Our proof is based on the techniques used in~\cite{TV06} for the same problem, in which the Fourier-analytic argument is over the group $\Z_p$ for some large enough $p$, rather than over the integers or over the real numbers.
The following claim is implicit in Corollary 7.16 in~\cite{TV06} and will be used in our computation.

\begin{claim}\label{claim:fordiff}
If $v_1, \ldots, v_n$ are distinct positive integers, then there exists a prime $p$ such that $p \geq v_i$ for all $i$, and
\[
 \frac{1}{p}\sum_{\xi \in \Z_p}\left[\prod_{i=1}^n|\cos(2\pi\xi \cdot v_i)|\right] \leq \frac{C}{n^{3/2}}.
\]
\end{claim}

We use Claim~\ref{claim:fordiff} to prove Theorem~\ref{thm:diff} which is a Markov chain version of Eq.~\eqref{eq:eqdiff}.

\begin{proof}[ of Theorem~\ref{thm:diff}]
Let $p$ be the prime in Claim~\ref{claim:fordiff}.
Note that by Fourier inversion,
\begin{align}
\Pr[f_1&(Y_1)v_1+f_2(Y_2)v_2+\cdots +f_n(Y_n)v_n = x_0] \nonumber \\
&\leq \Pr[f_1(Y_1)v_1+f_2(Y_2)v_2+\cdots +f_n(Y_n)v_n \equiv x_0 \bmod{p}] \nonumber \\
&= 
\frac{1}{p}\sum_{\xi \in \Z_p}\left|\exp\left(-\frac{2\pi i}{N} \xi \cdot x_0\right)\E\left[\exp\left(\frac{2 \pi i}{N}\xi\cdot (f(Y_1)v_1+f(Y_2)v_2+\cdots +f(Y_n)v_n)\right)\right]\right|. \label{eq:maindiff}
\end{align}
Let $T_j = A-(1-\lambda)E_{\mu}$ for all $j$, and let $u_i$ be the vector defined by $u_j(y) = \exp(2 \pi i (\xi \cdot f_j(y)v_j)/N)$.
Then the absolute value of the expectation inside the right-hand side of Eq.~\eqref{eq:maindiff} is bounded above by
\begin{multline*}
\left\| U_1 (T_1+(1-\lambda) E_{\mu}) U_2 (T_2+(1-\lambda) E_{\mu}) U_3\cdots U_{n-1} (T_{n-1}+(1-\lambda) E_{\mu}) U_{n} \mathbf{1} \right\|_{L_1(\mu)} \leq
\\
\sum_{s \in \{0, 1\}^{n-1}}\left(\prod_{j: s_j = 1}\lambda \right)\left( \prod_{j: s_j = 0} (1-\lambda) \right)\left(\prod_{j \in t(s)} \left|\cos(2\pi \xi \cdot v_j)\right|\right),
\end{multline*}
by Lemma~\ref{lem:holderapplicationrefine},
where for each $s \in \{0, 1\}^{n-1}$, we define $t(s)$ to be the set of indices $j$ such that $s_{j-1} = s_{j} = 0$, or $s_j = 0$ if $j = 1$ or $s_{j-1} = 0$ if $j = {k+1}$.
Thus by Claim~\ref{claim:fordiff}, we can upper bound on the right-hand side of Eq.~\eqref{eq:maindiff} by
\begin{align*}
\frac{1}{2\pi} \sum_{s \in \{0, 1\}^{n-1}} \left(\prod_{j: s_j = 1}\lambda\right)\left( \prod_{j: s_j = 0} (1-\lambda) \right) \frac{C}{(|t(s)|+1)^{3/2}} ,
\end{align*}
where the inequality also holds in the case that $|t(s)| = 0$.

As in the proof of Theorem~\ref{thm:highdim}, let $r: \{0, 1\}^{n-1} \rightarrow [n-1]$ be defined as
\[
r = \left|\{j: s_j = s_{j+1} = 0\}\right|,
\]
so that $r(s) \leq |t(s)|$ for all $s \in \{0, 1\}^{n-1}$, and let $\mathbf{s}$ be a random vector from $\{0, 1\}^{n-1}$ so that for each $s \in \{0, 1\}^{n-1}$
\[
\Pr[\mathbf{s} = s] = \left(\prod_{j: s_j = 1}\lambda\right)\left( \prod_{j: s_j = 0} (1-\lambda) \right).
\] 
By the definition of $r(\mathbf{s})$, we have
\[
\Pr[f(Y_1)v_1+f(Y_2)v_2+\cdots +f(Y_n)v_n = 0] \leq \frac{1}{2\pi}\E\left[\frac{C}{(r(\mathbf{s})+1)^{3/2}}\right]
\]
As before, let $r' = B(\lfloor(n/2\rfloor-1, (1-\lambda)^2)+1$.
Then because $r'$ is dominated by $r(s)$,
\begin{equation}
\E\left[\frac{C}{(r(\mathbf{s})+1)^{3/2}}\right] \leq  \E\left[\frac{C}{r'^{3/2}}\right] \leq \left(\E\left[\frac{C^{4/3}}{r'^2}\right]\right)^{3/4}\label{eq:compdiff},
\end{equation}
where again the second inequality follows by Jensen's inequality.
Finally, Claim~\ref{claim:combinarg} can be used to upper-bound the right-hand side of Eq.~\eqref{eq:compdiff}.
\end{proof}

\section{A Pseudorandom Generator for the Littlewood-Offord Problem}

In this section we prove Theorem~\ref{thm:prglo}.
As stated in the introduction, this theorem can be interpreted as proving the existence of a pseudorandom generator for the Littlewood-Offord problem.

We start by describing the construction of $D$.
Our construction will be based on expander graphs which we define as follows.
Given a $d$-regular graph $G = (V, E)$, let $A$ be the normalized adjacency matrix of $G$ and let $J$ be the matrix whose entries are all $1/|V|$.
We say that a family of $d$-regular graphs $\mathcal{G}$ is a family of expanders if for all graphs $G$ in the family, 
\[
\|A-J\|_{L_2(\mu) \rightarrow L_2(\mu)} \leq \lambda
\]
for some constant $\lambda$ bounded away from $1$, where $\mu$ is the vector whose entries are all $1/|V|$.
Note that when $G = (V, E)$ is $d$-regular, the stationary distribution is $\mu$, and the averaging operator is $J$.
Thus, $1-\|A-J\|_{L_2(\mu) \rightarrow L_2(\mu)}$ is also the spectral gap of the Markov chain that is a simple random walk on $G$.
It is well known that there exist infinite families of expander graphs of constant degree $d$ independent of the number of vertices (see for example,~\cite{LPS88} and~\cite{M88}).

Let $G = (\{-1, 1\}^{k}, E)$ be a $d$-regular graph from such a family so that $\|A-J\|_{L_2(\mu) \rightarrow L_2(\mu)} \leq \lambda$ for some constant $\lambda$ independent of $k$.
We let our set $D$ be the set of concatenations of the labels of walks of length $n/k$ on $G$, and thus $D$ has cardinality $2^{k+C_1n/k}$ for some constant $C_1$ independent of $n$ and $k$.

\begin{proof}[ of Theorem \ref{thm:prglo}]
Let $\mu$ be the uniform measure on $\{-1, 1\}^k$ and let $D$ be as defined above.
Then by Theorem~\ref{thm:esseen},
\begin{equation}\label{eq:prgstart}
\sup_{x_0 \in \R} \Pr_{\eps \sim D}[|\eps_1 v_1+\eps_2 v_2+\cdots +\eps_n v_n-x_0| \leq 1] \leq C \int_{-1}^{1} |\E[\exp(2 \pi i\xi (\eps_1 v_1+\cdots+\eps_nv_n))]|d\xi
\end{equation}
For each $j \in [n/k]$, let $T_j = A-(1-\lambda)J$ and let $u_j \in \R^{\{-1, 1\}^k}$ be the vector defined by 
\[
u_j(w) = \exp(2 \pi i \omega (w_{(j-1)k+1}v_{(j-1)k+1}+\cdots+w_{jk}v_{jk} ))
\]
and let $U_j = \diag(u_j)$.
Then $|\E[\exp(2 \pi i\xi (\eps_1 v_1+\cdots+\eps_nv_n))]|$ is bounded above by,
\begin{multline}\label{eq:prgint}
\left\| U_1 (T_1+(1-\lambda) J) U_2 (T_2+(1-\lambda) J) U_3\cdots U_{n/k-1} (T_{n/k-1}+(1-\lambda) J) U_{n/k} \mathbf{1} \right\|_{L_1(\mu)} \leq
\\
 \sum_{s \in \{0, 1\}^k}\left(\prod_{j: s_j = 1}\lambda\right)\left( \prod_{j: s_j = 0} (1-\lambda) \right)\left(\prod_{j \in t(s)} \left|\langle u_j, \mu \rangle\right|\right),
\end{multline}
where the inequality follows by Lemma~\ref{lem:holderapplicationrefine}, and for each $s \in \{0, 1\}^{n/k-1}$, we define $t(s)$ to be the set of indices $j$ such that $s_{j-1} = s_{j} = 0$, or $s_j = 0$ if $j = 1$ or $s_{j-1} = 0$ if $j = {n/k}$.

Note that $\langle u_j, \mu\rangle$ is the Fourier transform at $\xi$ of the random variable $w_{(j-1)k+1}v_{(j-1)k+1}+\cdots+w_{jk}v_{jk}$ where each coordinate of $w$ is uniformly random over the set $\{-1, 1\}$.
This brings us back to the original setting of completely independent random variables, and by Eq.~\eqref{eq:fincos}, it follows that
\[
\langle u_j, \mu\rangle = \prod_{\ell=1}^{k} \cos(2 \pi v_{(j-1)k+\ell} \xi).
\]
Thus by inserting the above in Eq.~\eqref{eq:prgint} we obtain and upper-bound on the right-hand side of Eq.~\eqref{eq:prgstart} of
\begin{multline*}
\frac{1}{2\pi}\sum_{s \in \{0, 1\}^{n/k-1}}\int_{-1}^{1}\left(\prod_{j: s_j = 1}\lambda \right)\left( \prod_{j: s_j = 0} (1-\lambda) \right)\left(\prod_{j \in t(s)} \prod_{\ell=1}^{k} \left|\cos(2 \pi v_{(j-1)k+\ell} \xi)\right|\right) d\xi
\leq \\
\frac{1}{2\pi}\sum_{s \in \{0, 1\}^{n/k-1}}\left(\prod_{j: s_j = 1}\lambda \right)\left( \prod_{j: s_j = 0} (1-\lambda) \right) \frac{C}{\sqrt{k(|t(\mathbf{s})|+1)}},
\end{multline*}
where the inequality follows from Claim~\ref{eq:imptv},
We proceed by using the same argument as in Lemma~\ref{thm:equal} starting from Eq.~\eqref{eq:intme}, which gives an upper bound of $C/\sqrt{k\cdot(n/k)} = C/\sqrt{n}$ as desired.
Finally, we obtain a construction of the desired size by letting $k = \sqrt{n}$.
\end{proof}

\bibliographystyle{alphaabbrv}
\bibliography{smallballexpsamp}

\appendix
\section{Proof of Lemma~\ref{lem:holderapplicationrefine}}\label{app:proof}
We prove Lemma~\ref{lem:holderapplicationrefine}, which as mentioned previously, is a straightforward adaptation of the proof of a Lemma from~\cite{NaorRR17}.
Before getting to the proof, we first state the following two claims.

\begin{claim}\label{clm:jbetween}
For all $k \ge 1$, matrices $R_1,\ldots,R_k \in \R^{N \times N}$, and distributions $\mu$ over $[N]$,
\[
\|R_1 E_{\mu} R_2 E_{\mu} \cdots E_{\mu} R_k \mathbf{1} \|_{L_1(\mu)}  \prod_{i=1}^k \|R_i \mathbf{1}\|_{L_1(\mu)} \; .
\]
\end{claim}
\begin{proof}
Notice that for any vector $v$, $E_{\mu}v = \E_{\mu}[v] \mathbf{1}$.  
The claim follows by noting that $\E_{\mu}[v] \leq \|v\|_{L_1(\mu)}$ and by induction.
\end{proof}

\begin{claim}\label{claimtwo}
Let $u_1, \ldots, u_{k+1} \in \C^{N}$ be so that $|(u_j)_i| \leq 1$ for all $i$ and $j$, let $U_j = \diag(u_j)$, and let $T_1, \ldots, T_k \in \R^{N \times N}$.  
Then,
\[
\left\|\left(\prod_{j=1}^{k} U_jT_j\right)U_{k+1}\mathbf{1}\right\|_{L_1(\mu)} 
\leq  
\prod_{j=1}^k \|T_j\|_{L_2(\mu) \rightarrow L_2(\mu)} \; ,
\]
\end{claim} 
\begin{proof}
By Jensen's inequality, the right-hand side is bounded above by
\[
\left\|\left(\prod_{j=1}^{k} U_jT_j\right)U_{k+1}\mathbf{1}\right\|_{L_2(\mu)} 
\]
and the claim follows by the definition of operator norm and the fact that $\|U_i\|_{L_2(\mu) \rightarrow L_2(\mu)} = \|u_i\|_{L_{\infty}(\mu)}$.
\end{proof}

\begin{claim}\label{claim:usesum}
Let $\mu \in R^N$ be a distribution, and let $E_{\mu}$ be the associated averaging operator.
Then for any $u \in \C^{N}$, 
\[
E_{\mu}\diag(u)E_{\mu} = \langle u, \mu \rangle E_{\mu}.
\]
\end{claim}
\begin{proof}
\[
(E_{\mu}\diag(u)E_{\mu})_{i, j} = \sum_{k \in [N]} (E_{\mu})_{i, k} u_{k} (E_{\mu})_{k, j} = \langle u, \mu \rangle \mu_j.\]
\end{proof}

\begin{proof}[ of Lemma~\ref{lem:holderapplicationrefine}]
For $j=1,\ldots,k$, let $T_{j, 0} = (1-\lambda)E_{\mu}$ and $T_{j, 1} = T_j$.
Let $U_{i}' = U_i$ if $i \in t(s)$, and $U_{i}' = I$ otherwise.
Then using the triangle inequality, the left-hand side of~\eqref{eq:splittingj} is at most
\begin{align}
\sum_{s \in \{0, 1\}^k} \left\|\left(\prod_{j=1}^k U_jT_{j, s_j}\right) U_{k+1} \mathbf{1}\right\|_{L_1(\mu)} 
=
\sum_{s \in \{0, 1\}^k} \left\|\left(\prod_{j}^k U_j'T_{j, s_j}\right) U_{k+1}' \mathbf{1}\right\|_{L_1(\mu)} \left(\prod_{j \in t(s)} \left|\langle \mu, u \rangle\right|\right) \label{eq:mainclaiminmono}
\end{align}
where the equality follows from Claim~\ref{claim:usesum} and the fact that $E_{\mu}^2 = E_{\mu}$.

Fix an $s \in \{0, 1\}^n$ and let $r_1, \ldots, r_{\ell} \in t(s)$ be the indices for which the ${r_i}$th coordinate of $s$ is $0$.
Then by Claim~\ref{clm:jbetween},
\begin{multline*}
\left\|\left(\prod_{j}^k U_j'T_{j, s_j}\right) U_{k+1}' \mathbf{1}\right\|_{L_1(\mu)} \leq 
(1-\lambda)\|U'_1T_{1}\cdots T_{{r_1-1}} U'_{r_1} \mathbf{1} \|_{L_1(\mu)} \\
(1-\lambda)\|U'_{{r_1+1}}T_{{r_1+1}}\cdots T_{{r_2-1}} U'_{r_2} \mathbf{1} \|_{L_1(\mu)} 
\cdots 
(1-\lambda)\|U'_{{r_{\ell}+1}}T_{{r_{\ell}+1}}\cdots T_{k} U'_{k+1} \mathbf{1} \|_{L_1(\mu)}  \; .
\end{multline*}
The claim now follows by applying Claim~\ref{claimtwo}.
\end{proof}

\end{document}